\documentclass[10pt]{article}
\usepackage{amsmath,amsthm,amsfonts,amssymb,amscd,  multicol, verbatim, enumerate, stackrel}
\usepackage{verbatim}
\usepackage{fullpage}
\input{xy}
\usepackage{xcolor}
\usepackage[english]{babel}
\usepackage{cite}
\usepackage{amsfonts}
\usepackage{latexsym}
\usepackage{amsthm}
\usepackage{amsopn}
\usepackage{amsmath}
\usepackage[all]{xy}

\usepackage{verbatim}
\usepackage{amssymb}
\usepackage{mathrsfs}
\usepackage{float}
\usepackage{fullpage}
\usepackage{amscd}
\usepackage{graphicx}

\newtheorem{thm}{Theorem}[section]
\newtheorem*{Problem}{Open Problem}
\newtheorem{lem}[thm]{Lemma}
\newtheorem{prop}[thm]{Proposition}
\newtheorem{cor}[thm]{Corollary}
\newtheorem{ex}[thm]{Example}
\theoremstyle{definition}
\newtheorem{defn}[thm]{Definition}
\theoremstyle{remark}
\newtheorem{rmk}[thm]{Remark}
\newtheorem*{rmk*}{Remark}

\begin{document}

\begin{center}

{\Large \bf Motzkin combinatorics in linear degenerations of the flag variety}
 \vspace{0.8cm} 


 Giovanni Cerulli Irelli\\

Dipartimento S.B.A.I. \\

 Sapienza-Universit\`a di Roma\\
 
via Scarpa, 10, \\

00199 Roma, Italy\\

{\em giovanni.cerulliirelli@uniroma1.it } \\

 \vspace{0.5cm}

Francesco Esposito\\

Dipartimento di Matematica\\ 

Universit\`a degli Studi di Padova\\

via Trieste, 63\\ 

35121 Padova, Italy\\

{\em esposito@math.unipd.it } \\

 \vspace{0.5cm}
 
 Mario Marietti  \\

Dipartimento  di Ingegneria Industriale e Scienze Matematiche\\

Universit\`a Politecnica delle Marche\\
 
 Via Brecce Bianche \\
 
 60131 Ancona,  Italy\\
 
{ \em m.marietti@univpm.it}\\
 
\end{center}

\vspace{1cm}

\begin{abstract}
We establish an explicit combinatorial/homological characterization of supports for linear degenerations of flag varieties. For such purpose, we introduce the concept of an excessive multisegment. It provides a new class of combinatorial objects counted by Motzkin numbers.
\end{abstract}

\vspace{.5cm}


\vspace{.5cm}
\section{Introduction}
\label{sez1}

Motzkin numbers have been introduced in \cite{M} and are known to count a variety of different objects. For an extensive list of such instances, one may consult \cite{DS, OEIS}; one of the basic types of objects counted by Motzkin numbers are the so-called Motzkin paths. Recently, Motzkin paths have made an unexpected appearance in a topological/algebraic context in the work of X. Fang and M. Reineke \cite{FR}. 

To put Fang and Reineke's work into context, let us recall that the flag variety for $\mathrm{SL}_{n+1} (\mathbb{C} )$ may be seen as a quiver Grassmannian. Let $Q$ be the equioriented quiver of type $A_n$, and $\mathrm{R}_\mathbf{d}  $ be the variety of representations of $Q$ of dimension vector $\mathbf{d}   = (n+1, \dots, n+1)\in \mathbb{N} ^n$, together with a basis of the underlying graded vector space $\mathbb{C} ^\mathbf{d}  $. Denote by $I$ the projective-injective indecomposable representation of $Q$, and consider the representation $I^{\oplus n+1}$ in $\mathrm{R}_\mathbf{d}  $. If $\mathbf{e} $ is the dimension vector $(1,2,\ldots , n)\in \mathbb{N} ^n$, then the quiver Grassmannian $\mathrm{Gr} _\mathbf{e}  (I^{\oplus n+1})$ is the variety of complete flags of $\mathbb{C} ^{n+1}$. This interpretation of the flag variety allows to consider quiver Grassmannians of more degenerate representations of $Q$ with the same dimension vector as degenerations of the flag variety. Such degenerations are defined and studied in \cite{CFFFR} (following the work on the degenerate flag variety initiated by A. Feigin \cite{F}); they are called linear degenerations of the flag variety. These correspond to quiver Grassmannians of representations in the so-called flat irreducible locus $\mathcal{U}_n$; it consists of all representations of $Q$ that degenerate to a fixed isomorphism class of representations which is the sum of all injective indecomposables and all projective indecomposables, all with multiplicity one except the projective-injective indecomposable, which has multiplicity two. 

Let $X_{\mathbf{e} , \mathbf{d}  }$ be the incidence variety of pairs $(A, B)$, where $A$ is a subrepresentation of $B\in \mathrm{R}_\mathbf{d}  $ of dimension $\mathbf{e} $. The projection on the second factor induces a proper map $\pi_{\mathbf{e}  , \mathbf{d}  }: X_{\mathbf{e}  , \mathbf{d}  }\rightarrow \mathrm{R}_{\mathbf{d}  }$, which is equivariant with respect to the action of the automorphism group  $\mathrm{G}_{\mathbf{d}  }=\mathrm{Aut}(\mathbb{C} ^{\mathbf{d}  })=(\mathrm{GL}_{n+1}(\mathbb{C} ))^n$ of the graded vector space $\mathbb{C} ^\mathbf{d}  $.
An important observation is that the linear degenerations of a fixed flag variety are the fibers of the restriction of the universal quiver Grassmannian $\pi_{\mathbf{e}  , \mathbf{d}  }: X_{\mathbf{e}  , \mathbf{d}  }\rightarrow \mathrm{R}_{\mathbf{d}  }$ over the flat irreducible locus $\mathcal{U}_{n}$. So one gets a proper $\mathrm{G}_{\mathbf{d}  }$-equivariant map $\pi_n : X_n \rightarrow \mathcal{U}_{n}$, where $X_n$ is a smooth algebraic variety.

By the Beilinson-Bernstein-Deligne-Gabber Decomposition Theorem \cite{BBD}, the push-forward of the constant sheaf through $\pi_n$ splits as a direct sum of IC complexes with finite-dimensional graded vector spaces as multiplicities:

\begin{equation}\label{decomposizione}
    (\pi_n)_{\ast}\underline{\mathbb{Q}   }_{X_n}[\dim X_n] = \bigoplus_{\mathcal{O}\subset \mathcal{U}_n} IC(\overline{\mathcal{O}}, \mathbb{Q}   ) \otimes V_{\mathcal{O}}.
\end{equation}

The variation of the cohomology of linear degenerations of the flag variety is thus reduced, modulo the computation of the coefficients $V_{\mathcal{O}}$, to the knowledge of stalks of the IC complexes of orbit closures. A $\mathrm{G}_{\mathbf{d}  }$-orbit $\mathcal{O}$ is said to be a {\em support} of the map $\pi_n$ provided that $V_{\mathcal{O}}\neq 0$.

Following this line of reasoning, Fang and Reineke study the topology of linear degenerations in \cite{FR}.
Their fundamental result is that the supports of $\pi_n$ are parametrized by Motzkin paths of length $n$. The parametrization of supports is given by maximazing certain functions over configurations attached to a specific Motzkin path (see Definition~\ref{FRmax}). As is the case in general for parametrizations, it is a nontrivial problem, given a specific $\mathrm{G}_\mathbf{d}  $-orbit, to determine if it is a support, and in case determine the Motzkin path it corresponds to.

In the light of Fang and Reineke's result and the techniques employed, some natural questions arise. 

The fact that Motzkin paths parametrize the supports is quite unexpected and surprising, and the method of proof of such parametrization does not shed sufficient light on this question. 
Fang and Reineke translate the problem of determining the coefficients $V_{\mathcal{O}}$ and the supports in an algebraic computation in the positive part of the quantized universal enveloping algebra for $A_n$. It consists in the expansion of a certain product of PBW basis elements in terms of the canonical basis of the quantum group. Such computation is highly involved, but nonetheless Fang and Reineke succeed in extracting enough information to determine which coefficients are nonzero (i.e., the supports) by means of the Knight--Zelevinsky multisegment duality (see \cite{K-Z}).

Starting from these considerations, we investigate an alternative approach to the problem. We work with the combinatorial counterparts of the geometric objects and carry out a study of these, highlighting the combinatorial motivations of the appearance of Motzkin paths. To this end, we  algebraically encode the combinatorial properties of Motzkin paths and show that this same structure can be found also in multisegments, which are combinatorial counterparts to quiver representations. 
From this point of view, we recover the parametrization of Fang and Reineke as the unique map compatible with the defined algebraic structures.
One of the benefits of such analysis is the characterization of the image of the parametrization in terms of structural conditions of the single multisegment; specifically, a multisegment is a support if and only if it has no triples of segments that are linked (see Definition \ref{triple linkate}). A multisegment satisfying this property is called {\em excessive}.

Although our study is purely combinatorial, there is an underlying geometric idea, which also motivates the name excessive.   It is based on the fact that the cohomologies of the fibers of the universal quiver Grassmannian 
$\pi_{\mathbf{e}  , \mathbf{d}  }: X_{\mathbf{e}  , \mathbf{d}  }\rightarrow \mathrm{R}_{\mathbf{d}  }$
are not independent objects, but are connected by specialization maps, since they are the cohomology stalks of a complex of sheaves with constructible cohomology on $\mathrm R_\mathbf{d}  $.
Thence, the possibility to estimate the coefficients in terms of these specialization maps. It was proved that these specialization maps are all surjective
for linear degenerations by M. Lanini and E. Strickland in \cite{LS}, and for quiver Grassmannians in type $A$ in \cite{CEFF}. This is a strong form of upper semicontinuity: in these cases the cohomology may only grow by degenerations. Thus the kernel of a specialization map measures the difference between the cohomologies of the two fibers involved; it embodies the new cohomology classes that appear in the more degenerate fiber and are not present in the more generic fiber of the degeneration in question. The space of new cohomology classes is the intersection of all kernels of specialization maps. Fixing the decomposition (\ref{decomposizione}), one sees that this space contains the graded vector space $V_\mathcal{O}  $. This gives a necessary condition for a $\mathrm G_\mathbf{d}  $-orbit $\mathcal{O}  $ to be a support, namely that the new space in the cohomology of its fiber be nonzero. The results in this work imply that the condition, unexpectedly,  is as well sufficient. When translating this condition in combinatorial terms, one encovers the definition of excessive multisegment, in other words a multisegment having successor-closed subdiagrams appearing in no generization of it. In this paper, this underlying geometric viewpoint and its combinatorial version are used only as heuristics; the combinatorial study we present is formally independent from its heuristic origin.

We mention the following open problem.
\begin{Problem}
Characterize the coefficients $V_\mathcal{O}  $ (as opposed to just supports) in terms of specialization maps or its combinatorial counterpart.
\end{Problem}

This paper is structured as follows. In Section~\ref{sez2}, we provide the Motzkin combinatorics that is needed in the rest of the work.
In Section~\ref{sez3}, we highlight algebraic structures on the set of multisegments. Even though most of this can be done in much greater generality, we restrict our treatment to the set $\mathcal R$ of multisegments of weight $\mathbf{d}  =(n+1, \ldots, n+1)$, which corresponds to the union over $n$ of the $\mathrm G_{\mathbf{d}  }$-orbits of $\mathrm R_{\mathbf{d}  }$. In particular, we introduce and discuss the concepts of concatenation and suspension of multisegments.
In Section~\ref{sez4}, we study a submonoid $\mathcal M$ of $\mathcal R$, which corresponds to the union over $n$ of the flat irreducible loci $\mathcal U_n$. It turns out that this locus is very particular also from the combinatorial point of view: the concatenation and the suspension behave better than on  $\mathcal R$, due to the fact that a restriction map may be defined.
In Section~\ref{sez5}, we introduce and study a new class of multisegments, which we call {\em excessive} and form a submonoid $\mathcal E$ of $\mathcal M$, so that we have $\mathcal E \subseteq \mathcal M \subseteq \mathcal R$. In particular, we prove that the excessive multisegments of fixed length provide a new type of objects counted by  Motzkin numbers.
In Section~\ref{sez6}, we relate the combinatorial framework developed in the previous sections to the geometry of quiver representations and, in particular, to the topology of linear degenerations of flag varieties. The main result is that the set of excessive multisegments is exactly the set of supports. As another application of the machinery developed in the previous sections, we provide a simple and explicit procedure to invert Fang and Reineke's parametrization.

\section{Motzkin combinatorics}
\label{sez2}

In this section, we give some properties of a certain monoid associated to Motzkin paths, which plays a role in the next sections in relation to multisegments.

\begin{defn}
Let $n\in \mathbb N$. A Motzkin path $\gamma$ of length $n$ is a map $\gamma : [0,n] \rightarrow \mathbb{N}$ such that:
\begin{enumerate}[(i)]
    \item $\gamma(0) = \gamma(n) = 0$;
    \item for all $k\in [0,n]$, one has $\gamma(k) \geq 0$;
    \item for all $k \in [0, n-1]$, one has $ |\gamma(k) - \gamma(k+1)| \leq 1$.
\end{enumerate}
We denote the set of Motzkin paths of length $n$  by  $\mathcal{M}  ot_n$ and the graded set $\coprod_{n\in \mathbb{N} }\mathcal{M}  ot_n$ of all Motzkin paths by $\mathcal{M}  ot$.
\end{defn}
It is usual to identify  a Motzkin path $\gamma$ in $\mathcal{M}  ot_n$ with the sequence $(\gamma(0), \gamma(1), \ldots, \gamma(n))$ as well as with the lattice path in the integer plane $\mathbb N \times \mathbb N$ from $(0, 0)$ to $(n, 0)$ connecting $(i-1, \gamma(i-1) )$ to $(i, \gamma(i) )$, for all $i\in [1,n]$. 

\begin{ex}
The Motzkin path  $\gamma=(0,0,1,0,1,2,1,2,1,0,0)$ in $\mathcal{M}  ot_{10}$ is identified with the following lattice path: 
$$
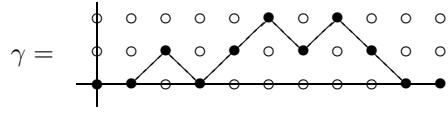
\gamma=
\vcenter{\xymatrix@C=2pt@R=2pt{
&&&&&&&&&&&\\
&{\circ}&{\circ}&{\circ}&{\circ}&{\circ}&*-{\bullet}&{\circ}&*-{\bullet}&{\circ}&{\circ}&{\circ}\\
&{\circ}&{\circ}&*-{\bullet}&{\circ}&*-{\bullet}&{\circ}&*-{\bullet}&{\circ}&*-{\bullet}&{\circ}&{\circ}\\
\ar@{-}[rrrrrrrrrrr] 
\ar@{-}'[rr]'[rrru]'[rrrr]'[rrrrru]'[rrrrrruu]'[rrrrrrru]'[rrrrrrrruu]'[rrrrrrrrru]'[rrrrrrrrrr]'[rrrrrrrrrrr]&
\bullet&*-{\bullet}&{\circ}&*-{\bullet}&{\circ}&{\circ}&{\circ}&{\circ}&{\circ}&*-{\bullet}&*-{\bullet}\\
&\ar@{-}[uuuu]&&&&&&&&&&
}}
$$
\end{ex}

The cardinality of the set $\mathcal{M}  ot_n$ is the {\em $n$-th Motzkin number} \cite{OEIS}. Notice that $\mathcal{M}  ot_0$ and $\mathcal{M}  ot_1$ are singletons.

Given two non-negative integers $p$ and $q$, we consider the map $ \ast : \mathcal{M}  ot_p \times \mathcal{M}  ot_q \rightarrow \mathcal{M}  ot_{p+q}$ defined as follows: for $\gamma\in \mathcal{M}  ot_p$ and $\delta \in  \mathcal{M}  ot_q $, the Motzkin path $\gamma \ast \delta $ is defined as
$$\gamma \ast \delta (i)= \left\{
\begin{array}{ll}
\gamma(i), & \text{ if $i\in[0,p]$,} \\
\delta(i-p), & \text{  if  $i\in[p,p+q]$.}
\end{array}
\right.$$
The induced binary operation 
$\ast : \mathcal{M}  ot \times \mathcal{M}  ot \rightarrow \mathcal{M}  ot$ on $\mathcal{M}  ot$ is called \emph{concatenation}.
\begin{ex}
If $\gamma \in \mathcal{M}  ot_{10}$ and $\delta\in\mathcal{M}  ot_{6}$ are the following lattice paths
$$
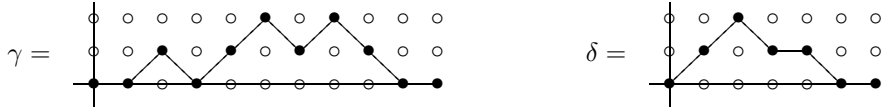
\begin{array}{ccc}
\gamma=
\vcenter{\xymatrix@C=2pt@R=2pt{
&&&&&&&&&&&\\
&{\circ}&{\circ}&{\circ}&{\circ}&{\circ}&*-{\bullet}&{\circ}&*-{\bullet}&{\circ}&{\circ}&{\circ}\\
&{\circ}&{\circ}&*-{\bullet}&{\circ}&*-{\bullet}&{\circ}&*-{\bullet}&{\circ}&*-{\bullet}&{\circ}&{\circ}\\
\ar@{-}[rrrrrrrrrrr] 
\ar@{-}'[rr]'[rrru]'[rrrr]'[rrrrru]'[rrrrrruu]'[rrrrrrru]'[rrrrrrrruu]'[rrrrrrrrru]'[rrrrrrrrrr]'[rrrrrrrrrrr]&
*-{\bullet}&*-{\bullet}&{\circ}&*-{\bullet}&{\circ}&{\circ}&{\circ}&{\circ}&{\circ}&*-{\bullet}&*-{\bullet}\\
&\ar@{-}[uuuu]&&&&&&&&&&
}}
&
\quad \quad \quad 
&
\delta=
\vcenter{\xymatrix@C=2pt@R=2pt{
&&&&&&&\\
&{\circ}&{\circ}&*-{\bullet}&{\circ}&{\circ}&{\circ}&{\circ}\\
&{\circ}&*-{\bullet}&{\circ}&*-{\bullet}&*-{\bullet}&{\circ}&{\circ}\\
\ar@{-}[rrrrrrr] 
\ar@{-}'[r]'[rru]'[rrruu]'[rrrru]'[rrrru]'[rrrrru]'[rrrrrr]'[rrrrrrr]&
*-{\bullet}&{\circ}&{\circ}&{\circ}&{\circ}&*-{\bullet}&*-{\bullet}\\
&\ar@{-}[uuuu]&&&&&&
}}
\end{array}
$$
then the concatenations $\gamma\ast \delta$ and $ \delta\ast\gamma$ in  $\mathcal{M}  ot_{16}$ are  the following lattice paths:
$$
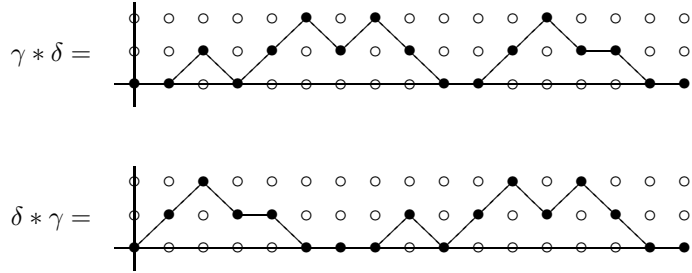
\gamma\ast\delta=
\vcenter{\xymatrix@C=2pt@R=2pt{
&&&&&&&&&&&
&&&&&&\\
&{\circ}&{\circ}&{\circ}&{\circ}&{\circ}&*-{\bullet}&{\circ}&*-{\bullet}&{\circ}&{\circ}&{\circ}
&{\circ}&*-{\bullet}&{\circ}&{\circ}&{\circ}&{\circ}\\
&{\circ}&{\circ}&*-{\bullet}&{\circ}&*-{\bullet}&{\circ}&*-{\bullet}&{\circ}&*-{\bullet}&{\circ}&{\circ}
&*-{\bullet}&{\circ}&*-{\bullet}&*-{\bullet}&{\circ}&{\circ}\\
\ar@{-}[rrrrrrrrrrrrrrrrr] 
\ar@{-}'[rr]'[rrru]'[rrrr]'[rrrrru]'[rrrrrruu]'[rrrrrrru]'[rrrrrrrruu]'[rrrrrrrrru]'[rrrrrrrrrr]'[rrrrrrrrrrr]'[rrrrrrrrrrrru]
'[rrrrrrrrrrrrruu]'[rrrrrrrrrrrrrru]'[rrrrrrrrrrrrrrru]'[rrrrrrrrrrrrrrrr]'[rrrrrrrrrrrrrrrrr]&
*-{\bullet}&*-{\bullet}&{\circ}&*-{\bullet}&{\circ}&{\circ}&{\circ}&{\circ}&{\circ}&*-{\bullet}&*-{\bullet} 
&{\circ}&{\circ}&{\circ}&{\circ}&*-{\bullet}&*-{\bullet}\\
&\ar@{-}[uuuu]&&&&&&&&&&
&&&&&&
}}
$$
$$
\delta\ast\gamma=
\vcenter{\xymatrix@C=2pt@R=2pt{
&&&&&&&
&&&&&&&&&&\\
&{\circ}&{\circ}&*-{\bullet}&{\circ}&{\circ}&{\circ}&
{\circ}&{\circ}&{\circ}&{\circ}&{\circ}&*-{\bullet}&{\circ}&*-{\bullet}&{\circ}&{\circ}&{\circ}\\
&{\circ}&*-{\bullet}&{\circ}&*-{\bullet}&*-{\bullet}&{\circ}&
{\circ}&{\circ}&*-{\bullet}&{\circ}&*-{\bullet}&{\circ}&*-{\bullet}&{\circ}&*-{\bullet}&{\circ}&{\circ}\\
\ar@{-}[rrrrrrrrrrrrrrrrr] 
\ar@{-}'[r]'[rru]'[rrruu]'[rrrru]'[rrrru]'[rrrrru]'[rrrrrr]'[rrrrrrr]'[rrrrrrrr]'[rrrrrrrrru]'[rrrrrrrrrr]'[rrrrrrrrrrru]'[rrrrrrrrrrrruu]'[rrrrrrrrrrrrur]'[rrrrrrrrrrrrrruu]'[rrrrrrrrrrrrrrru]'[rrrrrrrrrrrrrrrr]'[rrrrrrrrrrrrrrrrr]&
*-{\bullet}&{\circ}&{\circ}&{\circ}&{\circ}&*-{\bullet}&
*-{\bullet}&*-{\bullet}&{\circ}&*-{\bullet}&{\circ}&{\circ}&{\circ}&{\circ}&{\circ}&*-{\bullet}&*-{\bullet}\\
&\ar@{-}[uuuu]&&&&&&&&&&&&&&&&
}}
$$
\end{ex}

The set $\mathcal M ot$ has also a unary operation
$S: \mathcal{M}  ot \rightarrow \mathcal{M}  ot$, which we call suspension, defined as follows: given $\gamma\in \mathcal{M}  ot_n$,  the Motzkin path $S(\gamma)\in\mathcal{M}  ot_{n+2}$ satisfies 
$$S(\gamma) (i)= \left\{
\begin{array}{ll}
0, & \text{ if $i=0$ or $i=n+2$,} \\
\gamma(i-1) +1, & \text{ if $i\in[1,n+1]$.}
\end{array}
\right.$$
 
\begin{ex}
The suspension of the Motzkin path $\gamma\in\mathcal{M}  ot_{10}$ shown on the left is the Motzkin path $S(\gamma)\in\mathcal{M}  ot_{12}$ shown on the right:
$$
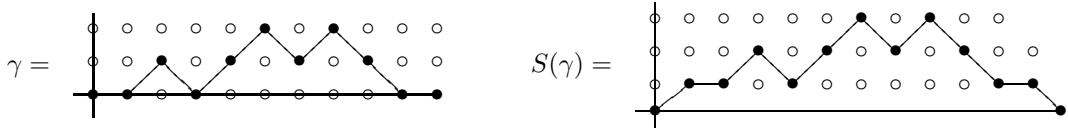
\begin{array}{cccc}
\gamma=
\vcenter{\xymatrix@C=2pt@R=2pt{
&&&&&&&&&&&\\
&{\circ}&{\circ}&{\circ}&{\circ}&{\circ}&*-{\bullet}&{\circ}&*-{\bullet}&{\circ}&{\circ}&{\circ}\\
&{\circ}&{\circ}&*-{\bullet}&{\circ}&*-{\bullet}&{\circ}&*-{\bullet}&{\circ}&*-{\bullet}&{\circ}&{\circ}\\
\ar@{-}[rrrrrrrrrrr] 
\ar@{-}'[rr]'[rrru]'[rrrr]'[rrrrru]'[rrrrrruu]'[rrrrrrru]'[rrrrrrrruu]'[rrrrrrrrru]'[rrrrrrrrrr]'[rrrrrrrrrrr]&
*-{\bullet}&*-{\bullet}&{\circ}&*-{\bullet}&{\circ}&{\circ}&{\circ}&{\circ}&{\circ}&*-{\bullet}&*-{\bullet}\\
&\ar@{-}[uuuu]&&&&&&&&&&
}}
&
&
&
S(\gamma)=
\vcenter{\xymatrix@C=2pt@R=2pt{
&&&&&&&&&&&&\\
&{\circ}&{\circ}&{\circ}&{\circ}&{\circ}&{\circ}&*-{\bullet}&{\circ}&*-{\bullet}&{\circ}&{\circ}\\
&{\circ}&{\circ}&{\circ}&*-{\bullet}&{\circ}&*-{\bullet}&{\circ}&*-{\bullet}&{\circ}&*-{\bullet}&{\circ}&{\circ}\\
&{\circ}
&*-{\bullet}&*-{\bullet}&{\circ}&*-{\bullet}&{\circ}&{\circ}&{\circ}&{\circ}&{\circ}&*-{\bullet}&*-{\bullet}&\\
\ar@{-}[rrrrrrrrrrrrr] &*-{\bullet}\ar@{-}'[ur]'[urr]'[rrruu]'[urrrr]'[urrrrru]'[urrrrrruu]'[urrrrrrru]'[urrrrrrrruu]'[urrrrrrrrru]'[urrrrrrrrrr]'[urrrrrrrrrrr]'[rrrrrrrrrrrr]&&&&&&&&&&&&*-{\bullet}\\
&\ar@{-}[uuuuu]&&&&&&&&&&&&
}}
\end{array}
$$
\end{ex}
A Motzkin path $\gamma$ in $\mathcal{M}  ot_n$ is not a concatenation of two smaller paths if and only if $\gamma(i)=0$ only for $i=0$ or $i=n$. In this case,    $\gamma(i)\geq 1$ for $i\in [1,n-1]$, and hence  $\gamma$ is the suspension  of a Motzkin path of length $n-2$, if $n\geq 2$.

Recall that a monoid $\mathfrak{M}$ (i.e., a set with an associative binary operation that has a neutral element) is {\em  graded} once it has a homomorphism $\mathfrak{M} \rightarrow \mathbb{N} $. The elements of $\mathfrak{M}$ that map to $n\in \mathbb{N} $  are said to be of degree $n$. We denote by $\mathfrak{M}_n$ the subset of $\mathfrak{M}$ containing all elements of degree $n$. We call an element $m$ of a monoid $\mathfrak{M}$ {\em primitive} if it is not a product of two elements in a nontrivial way.

\begin{prop}\label{Motzkin libero}
The set  $\mathcal{M}  ot$, equipped with the concatenation map, is a free graded monoid on the set $S(\mathcal{M}  ot) \coprod \mathcal{M}  ot_1$. Moreover, suspension is injective.
\end{prop}
\begin{proof}
It is straightforward that concatenation is associative and the unique element in $\mathcal{M}  ot_0$ is the identity. Moreover, every Motzkin path can be uniquely written as a concatenation of elements in  $S(\mathcal{M}  ot) \coprod \mathcal{M}  ot_1$. Injectivity of suspension follows directly from the definition.
\end{proof}

One may say that the following universal property of $\mathcal{M}  ot$ is the reason for the many occurrences of Motzkin numbers in algebraic combinatorics. 

\begin{prop}
\label{universalita}
Let $\mathfrak{M}= \coprod \mathfrak{M}_n$ be a free graded monoid with $\mathfrak{M}_1$ being a singleton and with a map  $S: \mathfrak{M} \rightarrow \mathfrak{M}$ of degree 2 (which we also call suspension). 
\begin{enumerate}[(i)]
    \item 
    \label{univ1}
    There exists a unique graded homomorphism of monoids
$\Phi: \mathcal{M}  ot \rightarrow \mathfrak{M}$ that commutes with the suspension maps. 
\item 
\label{univ2}
If the suspension of $\mathfrak{M}$ is injective and  its image is contained in the set of primitive elements of $\mathfrak{M}$, then $\Phi$ is injective.

\item 
\label{univ3}
The homomorphism  $\Phi$ is bijective if and only if $S(\mathfrak{M}) \coprod \mathfrak{M}_1$  is exactly the set of primitive elements of  $\mathfrak{M}$.
\end{enumerate}
\end{prop}
\begin{proof}
(\ref{univ1}). 
Let us inductively construct a graded homomorphism of monoids $\Phi=(\Phi_n): \mathcal{M}ot \to \mathfrak M$ commuting with suspension. 
There is only one map $\Phi_1$ from $\mathcal{M}  ot_1$ to $\mathfrak{M}_1$ since they are one-element sets. Let $n>1$ and suppose that we have already constructed $\Phi_i$ for every $i\in [0,n-1]$. Let us construct $\Phi_{n}$. Let $\gamma \in \mathcal{M}  ot_n$. 
By Proposition~\ref{Motzkin libero}, the path $\gamma $ can be uniquely written as a concatenation of Motzkin paths in  $S(\mathcal{M}  ot) \coprod \mathcal{M}  ot_1$. Thus one has exactly one choice for $\Phi_n(\gamma)$ since $\mathcal{M}  ot$ is free and its suspension is injective. Hence, the  unicity of $\Phi_n$ follows. All together, the maps $\Phi_n$, for $n\in \mathbb N$, determine a map $\Phi: \mathcal{M}ot \to \mathfrak M$ that is a homomorphism of monoids by the unicity of the factorization in $\mathcal{M}ot$; moreover, it is graded and commutes with suspension by construction. 

(\ref{univ2}). 
If the suspension of $\mathfrak{M}$ is injective with image contained in the set of primitives, then  injectivity of $\Phi_n$ follows from the facts that $\mathfrak{M}$ is free. 

(\ref{univ3}). 
The homomorphism $\Phi$ is bijective if and only if all primitive elements of $\mathfrak{M}$ are in the image of $\Phi$.  The set of primitive elements of $\mathcal{M} ot$ is $S(\mathcal{M}  ot) \coprod \mathcal{M}  ot_1$. By the properties of $\Phi$,  the image of a  primitive element is primitive and contained in $S(\mathfrak{M}) \coprod \mathfrak{M}_1$. Thus, if $\Phi$ is surjective, then $S(\mathfrak{M}) \coprod \mathfrak{M}_1$ is precisely the set of primitive elements. Viceversa, suppose that $S(\mathfrak{M}) \coprod \mathfrak{M}_1$ is the set of all primitive elements and, toward a contradiction, let $x\in\mathfrak{M}$ be a primitive element of minimal degree $n$ that does not belong to the image of $\Phi$. Then $x\in S(\mathfrak{M})$ and hence there exists  $y$ of degree $n-2$ such that $x=S(y)$. Then $y$ is a product of primitive elements of lower degree and hence $y$ belongs to the image of $\Phi$, say $y=\Phi(z)$. Then $x=S(y)=S(\Phi(z))=\Phi(S(z))$, which is a contradiction.
\end{proof}

\begin{rmk}
\label{cat}
The combinatorics of Catalan numbers may be handled in  analogous algebraic terms with a monoid $\mathcal Cat$  by replacing $S(\mathcal{M}  ot) \coprod \mathcal{M}  ot_1$ with simply $S(\mathcal Cat)$.
\end{rmk}

\section{Multisegments of weight $\mathbf{d}  $}
\label{sez3}

In this section, we highlight algebraic structures on the set of multisegments, which, to our knowledge, have not received attention prior to this work. Even though several of the following concepts can be introduced in much greater generality, we restrict our treatment to the multisegments of weight $\mathbf{d}  =(n+1, \ldots, n+1) \in \mathbb N^{n}$. Then, in the following sections, we give additional properties holding only on proper subsets of these multisegments.

Let us fix the notation. Given $i,j \in \mathbb N$, with $i\leq j$, we denote the interval of integers between $i$ and $j$, where $i$ and $j$ are included, by $[i,j]$. Following \cite{K-Z}, we call any such interval a {\em segment} and a multiset of segments a {\em multisegment}. The segment $[i,j]$ is said to begin at $i$ and end at $j$; it contains the column $k$ provided that $i \leq k < j$.

Let $n\in \mathbb N$, with $n\geq 1$. A multisegment over $[1,n]$ is a multisegment whose segments are contained in $[1,n]$. The {\em weight} of a multisegment $M$ over  $[1,n]$ is the sequence  $w(M)=(w_1(M), \ldots, w_n(M))$,  
where $w_k(M)$ is the number of segments of $M$ that contain $k$. The $k$-th column of a multisegment is the restriction of the multisegment to $[k,k+1]$. The $k$-th column of a multisegment is said to be {\em full} if such restriction contains the segment $[k, k+1]$ with multiplicity $\min \{ w(k), w(k+1)\}$; the multisgment is said to have $i$ {\em cuts} at column $k$ if such restriction contains the segment $[k, k+1]$ with multiplicity $\min \{ w(k), w(k+1)\} - i$.
\begin{ex}
As usual, we visualize a multisegment as an unordered bunch of strings. For example, the multisegment $\{[1,3], [1,1], [2,3], [2,4], [3,4], [3,3], [4,4], [4,4]\}$ is visualized as
$$
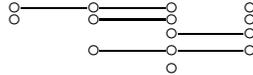
\xymatrix@R=0pt{
*{\circ}\ar@{-}[r]&*{\circ}\ar@{-}[r]&*{\circ}&*{\circ}\\
 *{\circ}             &*{\circ}\ar@{-}[r]&*{\circ}&*{\circ}\\
 &              &*{\circ}\ar@{-}[r]&*{\circ}\\
               &*{\circ}\ar@{-}[r]&*{\circ}\ar@{-}[r]&*{\circ}\\ 
               &&*{\circ}&
}
$$
It has weight $(2,3,5,4)$, one cut in the first column, no cuts in the second column and two cuts in the third column.
\end{ex}
For $n\geq 1$, we denote the set of multisegments over $[1,n]$ of weight $\mathbf{d}  $ by $\mathcal{R}_n$. 
Let $\mathcal R_0$ be the singleton  containing  the empty multisegment and denote the union $\coprod_{n\geq 0} \mathcal{R}_n$ by $\mathcal{R}$.

We now define two operations on the set $\mathcal{R}$, which we call concatenation and suspension. We denote by $\uplus$ the multiset sum: for example $\{x,x,x,y,z\} \uplus \{x,y,y\}=\{x,x,x,x,y,y,y,z\}$.

\begin{defn}
\label{3.1}
Let $M,N\in \mathcal{R}$, with $M\in \mathcal{R}_p$ and $N\in \mathcal{R}_q$. The concatenation of $M$ and $N$, denoted $M\ast N$,  is defined as follows. If $p\neq 0$ and $q\neq 0$, then  $M\ast N$ is the multisegment  in $\mathcal{R}_{p+q}$ containing the following segments:
\begin{enumerate}
        \item $[i,j]$ for each $[i,j]$ in $M$ with $j<p$,
        \item $[i+p,j+p]$ for each $[i,j]$ in $N$ with $i>1$,
        \item $[i_k, j_k+p]$ for each  $k=1,\ldots, p+q+1$, where 
        \begin{itemize}
            \item 
        $[i_1, p], \ldots [i_{p+q+1}, p]$, with $i_1 \geq i_2 \geq \cdots \geq i_{p+q+1} $,  are the segments of the multisegment $M \uplus \{ \underbrace{[1,p], [1,p], \ldots, [1,p]}_{q \text{ times}}  \}$  ending at $p$, 
        \item $[1, j_1], \ldots [1,j_{p+q+1}]$, with $j_1 \leq j_2 \leq \cdots \leq j_{p+q+1} $, are the segments of the multisegment $N \uplus \{ \underbrace{[1,q], [1,q], \ldots, [1,q]}_{p \text{ times}}  \}$  starting at $1$.
        \end{itemize}
        \end{enumerate}
        If $q=0$ (respectively,  $p=0$), then  $M\ast N=M$ (respectively, $M\ast N=N$).
       
\end{defn}

\begin{ex}
Let $M\in\mathcal{R}_3$ and $N\in\mathcal{R}_4$ be the following multisegments:
$$
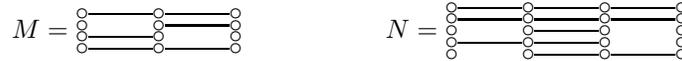
\begin{array}{ccc}
M=\vcenter{\xymatrix@R=0pt{
*{\circ}\ar@{-}[r]&*{\circ}\ar@{-}[r]&*{\circ}\\
*{\circ}&*{\circ}\ar@{-}[r]&*{\circ}\\
*{\circ}\ar@{-}[r]&*{\circ}&*{\circ}\\
*{\circ}\ar@{-}[r]&*{\circ}\ar@{-}[r]&*{\circ}
}}
&\textrm{ \quad \quad \quad  }&
N=
\vcenter{
\xymatrix@R=0pt{
*{\circ}\ar@{-}[r]&*{\circ}\ar@{-}[r]&*{\circ}\ar@{-}[r]&*{\circ}\\
*{\circ}\ar@{-}[r]&*{\circ}\ar@{-}[r]&*{\circ}\ar@{-}[r]&*{\circ}\\
*{\circ}&*{\circ}\ar@{-}[r]&*{\circ}&*{\circ}\\
*{\circ}\ar@{-}[r]&*{\circ}\ar@{-}[r]&*{\circ}&*{\circ}\\
*{\circ}&*{\circ}\ar@{-}[r]&*{\circ}\ar@{-}[r]&*{\circ}
}}
\end{array}
$$
The concatenation $M\ast N$ and the concatenation $N\ast M$ are obtained by first adding four rows to $M$ and three rows to $N$, and then by orderly ``gluing'' $M$ and $N$ along a new full column (for convenience of the reader we use the dotted lines in the new column):
$$
M\ast N=\vcenter
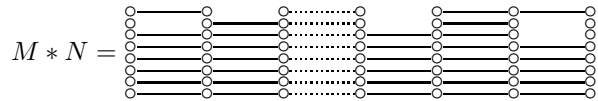
{\xymatrix@R=0pt{
*{\circ}\ar@{-}[r]&*{\circ}&*{\circ}
\ar@{..}[r]
&*{\circ}&*{\circ}\ar@{-}[r]&*{\circ}\ar@{-}[r]&*{\circ}\\
*{\circ}&*{\circ}\ar@{-}[r]&*{\circ}\ar@{..}[r]&*{\circ}&*{\circ}\ar@{-}[r]&*{\circ}&*{\circ}\\
*{\circ}\ar@{-}[r]&*{\circ}\ar@{-}[r]&*{\circ}\ar@{..}[r]&*{\circ}\ar@{-}[r]&*{\circ}\ar@{-}[r]&*{\circ}&*{\circ}\\
*{\circ}\ar@{-}[r]&*{\circ}\ar@{-}[r]&*{\circ}\ar@{..}[r]&*{\circ}\ar@{-}[r]&*{\circ}\ar@{-}[r]&*{\circ}\ar@{-}[r]&*{\circ}\\
*{\circ}\ar@{-}[r]&*{\circ}\ar@{-}[r]&*{\circ}\ar@{..}[r]&*{\circ}\ar@{-}[r]&*{\circ}\ar@{-}[r]&*{\circ}\ar@{-}[r]&*{\circ}\\
*{\circ}\ar@{-}[r]&*{\circ}\ar@{-}[r]&*{\circ}\ar@{..}[r]&*{\circ}\ar@{-}[r]&*{\circ}\ar@{-}[r]&*{\circ}\ar@{-}[r]&*{\circ}\\
*{\circ}\ar@{-}[r]&*{\circ}\ar@{-}[r]&*{\circ}\ar@{..}[r]&*{\circ}\ar@{-}[r]&*{\circ}\ar@{-}[r]&*{\circ}\ar@{-}[r]&*{\circ}\\
*{\circ}\ar@{-}[r]&*{\circ}\ar@{-}[r]&*{\circ}\ar@{..}[r]&*{\circ}\ar@{-}[r]&*{\circ}\ar@{-}[r]&*{\circ}\ar@{-}[r]&*{\circ}
}}
$$
and 
$$
N\ast M=\vcenter
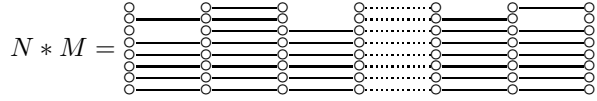
{\xymatrix@R=0pt{
*{\circ}&*{\circ}\ar@{-}[r]&*{\circ}&*{\circ}\ar@{..}[r]&*{\circ}&*{\circ}\ar@{-}[r]&*{\circ}\\
*{\circ}\ar@{-}[r]&*{\circ}\ar@{-}[r]&*{\circ}&*{\circ}\ar@{..}[r]&*{\circ}\ar@{-}[r]&*{\circ}&*{\circ}\\
*{\circ}&*{\circ}\ar@{-}[r]&*{\circ}\ar@{-}[r]&*{\circ}\ar@{..}[r]&*{\circ}\ar@{-}[r]&*{\circ}\ar@{-}[r]&*{\circ}\\
*{\circ}\ar@{-}[r]&*{\circ}\ar@{-}[r]&*{\circ}\ar@{-}[r]&*{\circ}\ar@{..}[r]&*{\circ}\ar@{-}[r]&*{\circ}\ar@{-}[r]&*{\circ}\\
*{\circ}\ar@{-}[r]&*{\circ}\ar@{-}[r]&*{\circ}\ar@{-}[r]&*{\circ}\ar@{..}[r]&*{\circ}\ar@{-}[r]&*{\circ}\ar@{-}[r]&*{\circ}\\
*{\circ}\ar@{-}[r]&*{\circ}\ar@{-}[r]&*{\circ}\ar@{-}[r]&*{\circ}\ar@{..}[r]&*{\circ}\ar@{-}[r]&*{\circ}\ar@{-}[r]&*{\circ}\\
*{\circ}\ar@{-}[r]&*{\circ}\ar@{-}[r]&*{\circ}\ar@{-}[r]&*{\circ}\ar@{..}[r]&*{\circ}\ar@{-}[r]&*{\circ}\ar@{-}[r]&*{\circ}\\
*{\circ}\ar@{-}[r]&*{\circ}\ar@{-}[r]&*{\circ}\ar@{-}[r]&*{\circ}\ar@{..}[r]&*{\circ}\ar@{-}[r]&*{\circ}\ar@{-}[r]&*{\circ}
}}
$$
\end{ex}
\begin{rmk}
\label{inclusi uno nellaltro}
Let $ M\in \mathcal{R}_p$ and $N\in \mathcal{R}_q$. The segments of $M\ast N$ containing $\{p,p+1\}$ are linearly ordered by inclusion: the first such segment is obtained by putting together the shortest segment of $M$ among those ending at $p$ with the shortest segment of $N$ among those starting at $1$, and so on.
\end{rmk}

\begin{defn}
\label{sospensione}
Let $M \in \mathcal{R}_n$. The suspension of $M$, denoted $S(M)$, is defined as follows. If $n>0$, then $S(M)$ is the multisegment of $\mathcal{R}_{n+2}$ containing the following segments:
\begin{enumerate}
    \item $[i+1, j+1]$ for each $[i, j]\in M$ with $1 < i \leq j < n$,
    \item $[1, j+1]$ for each $[1, j]\in M$ with $j < n$, 
    \item $[i+1, n+2]$ for each $[i, n]\in M$ with $i > 1$,
    \item $[1, n+2]$ for each $[1,n]\in M$, 
    \item
    \label{sospensione5}
    $[1,1]$, $[2, n+2]$, $[1, n+1]$, and $[n+2, n+2]$.
\end{enumerate}
If $n=0$, the suspension of the empty multisegment  is the multisegment $\{[1,1], [2, 2], [1,2]\}$ of $\mathcal{R}_{2}$.

\end{defn}

\begin{ex} The suspension of a multisegment $N$ is obtained by adding two columns and two rows (which we depict in dotted lines for convenience of the reader):
$$
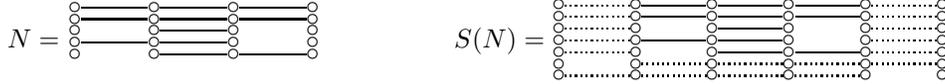
\begin{array}{ccc}
N=
\vcenter{
\xymatrix@R=0pt{
*{\circ}\ar@{-}[r]&*{\circ}\ar@{-}[r]&*{\circ}\ar@{-}[r]&*{\circ}\\
*{\circ}\ar@{-}[r]&*{\circ}\ar@{-}[r]&*{\circ}\ar@{-}[r]&*{\circ}\\
*{\circ}&*{\circ}\ar@{-}[r]&*{\circ}&*{\circ}\\
*{\circ}\ar@{-}[r]&*{\circ}\ar@{-}[r]&*{\circ}&*{\circ}\\
*{\circ}&*{\circ}\ar@{-}[r]&*{\circ}\ar@{-}[r]&*{\circ}\\
&&&
}}
&
\quad \quad \quad
&
S(N)=
\vcenter{
\xymatrix@R=0pt{
*{\circ}\ar@{..}[r]&*{\circ}\ar@{-}[r]&*{\circ}\ar@{-}[r]&*{\circ}\ar@{-}[r]&*{\circ}&*{\circ}\ar@{..}[l]\\
*{\circ}\ar@{..}[r]&*{\circ}\ar@{-}[r]&*{\circ}\ar@{-}[r]&*{\circ}\ar@{-}[r]&*{\circ}&*{\circ}\ar@{..}[l]\\
*{\circ}\ar@{..}[r]&*{\circ}&*{\circ}\ar@{-}[r]&*{\circ}&*{\circ}&*{\circ}\ar@{..}[l]\\
*{\circ}\ar@{..}[r]&*{\circ}\ar@{-}[r]&*{\circ}\ar@{-}[r]&*{\circ}&*{\circ}&*{\circ}\ar@{..}[l]\\
*{\circ}\ar@{..}[r]&*{\circ}&*{\circ}\ar@{-}[r]&*{\circ}\ar@{-}[r]&*{\circ}&*{\circ}\ar@{..}[l]\\
*{\circ}&*{\circ}\ar@{..}[r]&*{\circ}\ar@{..}[r]&*{\circ}\ar@{..}[r]&*{\circ}\ar@{..}[r]&*{\circ}\\
*{\circ}\ar@{..}[r]&*{\circ}\ar@{..}[r]&*{\circ}\ar@{..}[r]&*{\circ}\ar@{..}[r]&*{\circ}&*{\circ}
}}
\end{array}
$$
\end{ex}

We say that a pair of segments $(A,B)$ is {\em linked} provided that the segment  $A$ begins strictly before the segment  $B$, the segment $B$ ends strictly after the segment  $A$, and  $A\cup B$ is a segment. This is a slight modification of the classical definition by Zelevinsky (see \cite{Z1} or \cite{Z2}). 
\begin{ex}
In the following table, we give some examples of linked and not linked pairs of multisegments in order to illustrate the various possibilities.  
\begin{center}
\begin{tabular}{|c|c|c|c|c|}
\hline
$$
\xymatrix@C=5pt@R=0pt{
\ar@{-}^A[rrr]&&&&\\
&\ar@{-}_B[rrr]&&&
}
$$
&
$$
\xymatrix@C=5pt@R=0pt{
\ar@{-}^B[rrr]&&&&\\
&\ar@{-}_A[rrr]&&&
}
$$
&
$$
\xymatrix@C=5pt@R=0pt{
\ar@{-}^A[rr]&&*{\circ}&&\\
&&*{\circ}\ar@{-}_B[rr]&&
}
$$
&
$$
\xymatrix@C=5pt@R=0pt{
\ar@{-}^A[rr]&&*{\circ}&*{\circ}\ar@{-}_B[rr]&&
}
$$
&
$$
\xymatrix@C=5pt@R=0pt{
\ar@{-}^A[rr]&&*{\circ}&*{\circ}&*{\circ}\ar@{-}_B[rr]&&
}
$$
\\
\hline
$(A,B)$&$(A,B)$&$(A,B)$&$(A,B)$&$(A,B)$\\
linked&not linked&linked&linked&
not linked
\\\hline
\end{tabular}
\end{center}
\end{ex}

\begin{defn}
A multisegment $M$ is said to have a special full column at $k$ if the $k$-th column of $M$ is full, and there is no linked pair $(A,B)$ of segments of $M$ that contains the column $k$ (i.e., given any two segments $A$ and $B$ containing the column $k$, either $A\subseteq B$ or $B\subseteq A$).
\end{defn}

\begin{lem}
\label{iniettivo R}
    The concatenation sends injectively $\mathcal R_p \times \mathcal R_q$ to the subset of multisegments in $\mathcal R_{p+q}$  having a special full column at $p$.
\end{lem}
\begin{proof}
  Let $M\in \mathcal R_p$ and $N \in \mathcal R_q$. By construction (see Remark~\ref{inclusi uno nellaltro}),  $M\ast N$ has a special full column at $p$. The injectivity is straightforward.
\end{proof}

\begin{rmk}
We note that not every multisegment in $\mathcal R_{n}$  having a special full column can be obtained as a non-trivial concatenation, as the multisegments $B$ and $B'$ of Remark~\ref{noliberono}.
\end{rmk}

\begin{prop}
\label{libero graduato R}
The  concatenation $\ast: \mathcal{R}\times \mathcal{R} \rightarrow \mathcal{R}$ makes $\mathcal{R}$ into a graded monoid with the empty multisegment as identity element. The suspension map is an injective self-map of $\mathcal{R}$ with image contained in the subset of primitive elements.
\end{prop}
\begin{proof}
The first statement is straightforward. From the definition, it is straightforward that the suspension map $S$ is injective.  
Let $M\in \mathcal R_n$. The first and the last columns of $S(M)$ are not full by Definition~\ref{sospensione}(\ref{sospensione5}). Furthermore, any other column of $S(M)$, even when full, cannot be special since it would be contained in the segments $[2,n+2]$ and $[1,n+1]$ of $S(M)$. Hence, $S(M)$ has no special full columns, and thus is primitive by Lemma~\ref{iniettivo R}.
\end{proof}

\begin{rmk}
\label{noliberono}
We note that the monoid $\mathcal R$ is not free. For example, let us consider the following multisegments of $\mathcal R$:
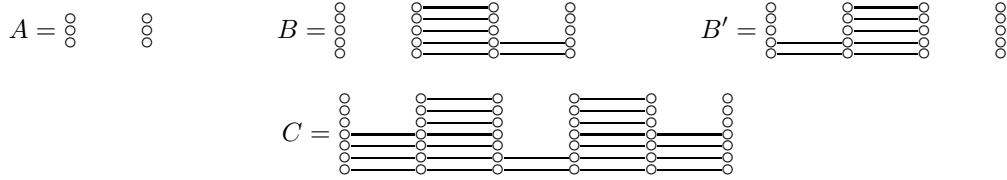
$$
A=\vcenter{\xymatrix@R=0pt{
*{\circ}&*{\circ}\\
*{\circ}&*{\circ}\\
*{\circ}&*{\circ}\\
}}\ \ 
\quad
\quad
\quad
\quad
B=\vcenter{\xymatrix@R=0pt{
*{\circ}&*{\circ}\ar@{-}[r]&*{\circ}&*{\circ}\\
*{\circ}&*{\circ}\ar@{-}[r]&*{\circ}&*{\circ}\\
*{\circ}&*{\circ}\ar@{-}[r]&*{\circ}&*{\circ}\\
*{\circ}]&*{\circ}\ar@{-}[r]&*{\circ}\ar@{-}[r]&*{\circ}\\
*{\circ}&*{\circ}\ar@{-}[r]&*{\circ}\ar@{-}[r]&*{\circ}
}}\ \ 
\quad
\quad
\quad
\quad
B'=\vcenter{\xymatrix@R=0pt{
*{\circ}&*{\circ}\ar@{-}[r]&*{\circ}&*{\circ}\\
*{\circ}&*{\circ}\ar@{-}[r]&*{\circ}&*{\circ}\\
*{\circ}&*{\circ}\ar@{-}[r]&*{\circ}&*{\circ}\\
*{\circ}\ar@{-}[r]&*{\circ}\ar@{-}[r]&*{\circ}]&*{\circ}\\
*{\circ}\ar@{-}[r]&*{\circ}\ar@{-}[r]&*{\circ}&*{\circ}
}} \ \ 
$$

$$
C=\vcenter{\xymatrix@R=0pt{
*{\circ}&*{\circ}\ar@{-}[r]&*{\circ}&*{\circ}\ar@{-}[r]&*{\circ}&*{\circ}\\
*{\circ}&*{\circ}\ar@{-}[r]&*{\circ}&*{\circ}\ar@{-}[r]&*{\circ}&*{\circ}\\
*{\circ}&*{\circ}\ar@{-}[r]&*{\circ}&*{\circ}\ar@{-}[r]&*{\circ}&*{\circ}\\
*{\circ}\ar@{-}[r]&*{\circ}\ar@{-}[r]&*{\circ}&*{\circ}\ar@{-}[r]&*{\circ}\ar@{-}[r]&*{\circ}\\
*{\circ}\ar@{-}[r]&*{\circ}\ar@{-}[r]&*{\circ}&*{\circ}\ar@{-}[r]&*{\circ}\ar@{-}[r]&*{\circ}\\
*{\circ}\ar@{-}[r]&*{\circ}\ar@{-}[r]&*{\circ}\ar@{-}[r]&*{\circ}\ar@{-}[r]&*{\circ}\ar@{-}[r]&*{\circ}\\
*{\circ}\ar@{-}[r]&*{\circ}\ar@{-}[r]&*{\circ}\ar@{-}[r]&*{\circ}\ar@{-}[r]&*{\circ}\ar@{-}[r]&*{\circ}
}}\ \ 
$$
The multisegment $C$ has two 
decompositions into primitive elements, namely $C=A\ast B$ and $C = B' \ast A$.
\end{rmk}

\section{The submonoid $\mathcal M$}
\label{sez4}
In this section, we study a submonoid of $\mathcal R$, which corresponds to the union of flat irreducible loci of all $\mathrm{R}_\mathbf{d}  $'s, i.e. the spaces over which  the linear degenerations of flag varieties live.
It turns out that this locus is very particular also from the combinatorial point of view.

For $n\geq 1$, let $\mathcal M_n$ be the subset of  $\mathcal{R}_n$  of those multisegments having at most one cut per column. We let $\mathcal{M}$ be the union $\coprod_{n\geq 0} \mathcal{M}_n$, where $\mathcal{M} _0$ is the singleton  containing  the empty multisegment.

We now introduce the restrictions of multisegments. Loosely speaking, the left (respectively, right) restriction of a multisegment $M$ consists of keeping only the first (respectively, last) $k$ vertices of $M$ and forgetting the appropriate number of full segments in order to obtain a multisegment in $\mathcal M_k$.
\begin{defn}
\label{restrizione}
Let $M\in \mathcal M_n$ and $k\in \mathbb N$ with $1\leq k\leq n$. The left restriction of $M$ to $k$, denoted $_{(k)}M$, is the multisegment in $\mathcal M_{k}$ 
$$ _{(k)}M = M' \setminus \{\underbrace{[1,k], [1,k], \ldots, [1,k]}_{(n-k) \text{ times}}  \} $$
where $M'$ is the multisegment containing the following segments:
\begin{enumerate}
        \item $[i,j]$ for each $[i,j]$ in $M$ with $j\leq k$,
        \item $[i,k]$ for each $[i,j]$ in $M$ with $i\leq k$ and $j>k$.
\end{enumerate}
The right restriction of $M$ to $k$, denoted $M_{(k)}$, is the multisegment in $\mathcal M_{k}$ 
$$ M_{( k)} = M'' \setminus \{\underbrace{[1,k], [1,k], \ldots, [1,k]}_{(n-k) \text{ times}}  \} $$
where $M''$ is the multisegment containing the following segments:
\begin{enumerate}
        \item $[i-n+k,j-n+k]$ for each $[i,j]$ in $M$ with $i\geq n-k+1$,
        \item $[1,j-n+k]$ for each $[i,j]$ in $M$ with $i < n-k+1$ and $j \geq n-k+1$.
\end{enumerate}
\end{defn}
Notice that Definition~\ref{restrizione} is well defined. As for $_{(k)}M$, the multisegment $M'$ has length $k$ and weight $(n+1, \ldots, n+1)$ and, since $M\in \mathcal M$, it has at most one cut per column and hence contains at least $n-k$ copies of the segment $[1,k]$. Once $n-k$ copies of $[1,k]$ are removed, one obtains a multisegment in $\mathcal M_k$. An analogous argument holds for $M_{( k)}$.

Also notice that Definition~\ref{restrizione} cannot be given for $M$ in $\mathcal R$ since there would not be a natural choice for the multisegments to remove.

\begin{rmk}
\label{vice versa}
We notice that, if a multisegment $M\in \mathcal{M}_n$ has a special full column at $p$, then
$$M= \; _{(p)}M \ast M_{(n-p)}.$$
Moreover, in this case, $_{(p)}M$ has a special full column if and only if $M$ has also another special full column at $p'$, with $p'<p$.
Analogously for  $ M_{ (n-p)}$. Note that the statement does not continue to hold when $\mathcal M$ is replaced by $\mathcal R$ (see Remark~\ref{noliberono}).
\end{rmk}

\begin{lem}
\label{iniettivo}
The concatenation sends bijectively $\mathcal M_p \times \mathcal M_q$ to the subset of multisegments in $\mathcal M_{p+q}$  having a special full column at $p$.
\end{lem}
\begin{proof}
  By Lemma~\ref{iniettivo R}, it suffices to prove surjectivity, which follows by Remark~\ref{vice versa}.
\end{proof}

 The following proposition illustrates algebraic properties of the submonoid $\mathcal M$.  
\begin{prop}
\label{libero graduato}
The following statements hold:
\begin{enumerate}[(i)]
    \item \label{3}  the subset $\mathcal{M} $ of $\mathcal{R}$ is a free graded submonoid  whose primitive elements are the multisegments in $\mathcal M$ with no special full columns;
    \item \label{4} the suspension $S(M)$ is in the submonoid $\mathcal M$ if and only if $M\in \mathcal M$.
\end{enumerate}
\end{prop}
\begin{proof}
(\ref{3}) Let $M\in \mathcal{M}$. By iterated application of Lemma~\ref{iniettivo}, which is permitted by the second part of Remark~\ref{vice versa}, the multisegment $M$ may be uniquely written as
$$
M = M_1 \ast M_2 \ast \cdots 
\ast M_r,
$$
with the multisegments $M_i$ primitive. 
Hence the statement holds.

(\ref{4}) It is straightforward from the definition of the suspension map that $S(M)$ has at most one cut per column if and only if $M$ has at most one cut per column. Thus the statement follows.
\end{proof}

\begin{rmk}\label{rmksospensione}
Note that a multisegment $M$ in $\mathcal{M} _n$ is a suspension if and only if it contains the segments $[1,n-1]$ and $[2,n]$.
\end{rmk}

This implies the following.

\begin{prop}
\label{unico morfismo}
There exists a unique injective homomorphism of graded monoids from $\mathcal{M}  ot$ to $\mathcal{M}$ commuting with suspension. 
\end{prop}
\begin{proof}
Recall that $\mathcal M_1$ is a  one-element set.

By Proposition~\ref{libero graduato}, the monoid  $\mathcal M$ is  free graded and has an injective suspension map  $S: \mathcal{M} \rightarrow \mathcal{M}$ of degree 2 whose  image is contained in the set of primitive elements of $\mathcal{M}$. Thus the statement follows by Proposition~\ref{universalita}. 
\end{proof}

\section{Excessive multisegments}
\label{sez5}
In this section, we introduce and study a new class of multisegments, which we call {\em excessive}. We prove that the set of excessive multisegments is exactly the image of the map  of Proposition~\ref{unico morfismo}. In particular, we show that, for each $n$ in $\mathbb N$, the excessive multisegments over $[1,n]$ are counted by the $n$-th Motzkin number.

\begin{defn}
\label{triple linkate}
A pair of segments $(A,B)$ is said to be {\em quasi-linked} provided that either
\begin{enumerate}[(i)]
    \item $(A,B)$ is linked, or
    \item $a+1 = b-1$, 
\end{enumerate}
where $a$ is the end of $A$ and $b$ is the beginning of $B$.

A triple of segments $(A,B,C)$ is said to be {\em linked} provided that
\begin{enumerate}[(i)]
    \item $(A,B)$ is linked,
    \item $(B,C)$ is linked,
    \item $(A,C)$ is quasi-linked.
\end{enumerate}
\end{defn}

\begin{defn}
Let $n>0$. A multisegment $M\in \mathcal{M} _n$ is said to be excessive if there are no linked triples of its segments. The set of excessive multisegments over $[1,n]$ is denoted by $\mathcal{E} _n$. Moreover, we let $\mathcal{E} _0=\mathcal{M} _0$ and $\mathcal{E}  = \coprod_{n\geq 0} \mathcal{E} _n$. 
\end{defn}

\begin{lem}
\label{eccesso}
The following statements hold.
\begin{enumerate}[(i)]
    
    \item 
    \label{eccesso1}
    Let $M\in \mathcal R_p$ and $N\in \mathcal R_q$. The concatenation $M\ast N$ is excessive if and only if both $M$ and $N$ are.
    \item 
    \label{eccesso2} Let $M\in \mathcal R_p$. The suspension $S(M)$ is excessive if and only if $M$ is.
    \item 
    \label{eccesso3} The subset $\mathcal{E} $ of $\mathcal R$ is a graded free submonoid stable by suspension whose set of primitives is $S(\mathcal{E} )\coprod \mathcal{E} _1 $.
\end{enumerate}
\end{lem}

\begin{proof}
(\ref{eccesso1}). 
    Since the assertion  is trivial when $0\in\{p,q\}$, we suppose $0\notin\{p,q\}$.
    
    It is straightforward to see that $M \notin \mathcal E$ (as well as $N \notin \mathcal E$) implies $M\ast N \notin \mathcal E$.  
    
    Let $M,N\in \mathcal E$. Toward a contradiction, suppose that  $M\ast N$  has a linked triple $(A,B,C)$.
    Since $M\ast N$ has a special full column at $p$ and two almost linked segments are not contained in one another, at most one segment among $A$, $B$, and $C$ contain $\{p,p+1\}$. Since a segment ending before $p$ cannot be almost linked with a segment starting after $p+1$, the hypothesis  $M,N\in \mathcal E$ implies that exactly one segment among $A$, $B$, and $C$ contains $\{p,p+1\}$. Thus, either
    \begin{enumerate}
        \item $A$ and $B$ end before $p$, and $\{p,p+1\}\subseteq C$, or (symmetrically)
        \item $B$ and $C$ start after $p+1$, and $\{p,p+1\}\subseteq A$.
    \end{enumerate}
    In the first case, we have that $(A,B,\  _{(p)}\! C)$ is a linked triple of $M$ (recall Definition~\ref{restrizione}), which is a contradiction since $M \in \mathcal E$. In the second case, we may argue analogously.

   (\ref{eccesso2}).  
   Since the assertion holds for $p=0$, we suppose $p>0$.
   
   It is straightforward to see that $M \notin \mathcal E$  implies $S(M) \notin \mathcal E$. 
   
   Let $M \in \mathcal E$. Toward a contradiction, suppose that  $S(M)$  has a linked triple $(A,B,C)$. If $A=[1,1]$, then $B=[2,p+2]$, since $[2,p+2]$ is the unique segment of $S(M)$ that is linked to $[1,1]$, and $C$ is a segment starting from 3, since these are the only segments that are almost linked to $[1,1]$. This is a contradiction since $C$ would be contained in $B$. Thus $[1,1]$ is not part of any linked triple, and analogously for $[p+2, p+2]$. Moreover, the segment $[2,p+2]$ is linked only to segments starting at 1, and is not quasi-linked to any segment; thus it is not contained in any linked triple of $S(M)$. The same holds for $[1,p+1]$. Thus, the linked triple $(A,B,C)$ of $S(M)$ corresponds to a linked triple of $M$, which is a contradiction.

   (\ref{eccesso3}). By (\ref{eccesso1}) and (\ref{eccesso2}), the subset $\mathcal{E} $ of $\mathcal{M} $ is a graded submonoid stable by suspension. 
   
   Let us prove that the  set of primitives of $\mathcal{E} $ is $S(\mathcal{E} ) \coprod \mathcal{E} _1$. The unique element of $\mathcal{E} _1$ is clearly primitive. Let $M\in \mathcal{E} _n$, with $n>0$. By Lemma~\ref{iniettivo} and (\ref{eccesso1}), it is enough to show that if $M$ has no special full columns then $[2,n]\in M$. Indeed, the symmetric argument would show that $[1,n-1]$ is a segment of $M$, and hence, by Remark~\ref{rmksospensione},  $M$ is of the form $S(M')$, for a certain $M'\in\mathcal{E} _{n-2}$. Let us proceed along these lines by induction on $n$.
   
   Let  $M\in\mathcal{E} _n$ and suppose that $M$ has no special full columns. Then the first column must contain a cut, or else it would be special full. It follows that $M$ contains a segment $J$ of the form $J=[2,j]$, for a unique $j$ such that $2<j\leq n$, since $M\in \mathcal M$. We need to prove that $j=n$. Toward a contradiction, suppose 
    $j<n$. Observe that the segments of $M$ starting at $1$ and different from $[1,1]$ cannot all contain $J$,
   because in such case the columns of $J$ would be special full for $M$. Let $R=[1,r]$ be the biggest segment of $M$ among those that start at $1$ and do not contain $J$.  
   
   One has  $M= \; _{(p)}M \ast M_{ (n-p)}$ by Remark~\ref{vice versa}, with  $M_{  (n-p)}\in \mathcal{E} _{n-p}$  by (\ref{eccesso1}).
   The induction hypothesis implies $R'=[r+1, n]\in M$.
   This would give rise to the linked triple $(R, J, R')$ of segments of $M$, which contradicts the assumption that $M$ is excessive. Thus $j=n$ and the assertion follows.
\end{proof}

\begin{rmk}
\label{vice}
Analogously as in Remark~\ref{vice versa}, we note that if an excessive multisegment $E\in \mathcal{E}_n$ has a special full column at $p$, then
$$E=\; _{(p)}E \ast E_{ (n-p)},$$
where both $_{(p)}E$ and $E_{ (n-p)}$ are excessive by Lemma~\ref{eccesso}(\ref{eccesso1}). If instead $E$ has no special full columns, then $E$ is the suspension of its double restriction  $(_{(n-1)}E)_{(n-2)}$, which coincides with $\; _{(n-2)}(E_{(n-1)}) $.
\end{rmk}

\begin{thm}
\label{teorema}
There exists a unique graded monoid isomorphism $\mathcal{M}  ot \rightarrow \mathcal{E} $ which commutes with suspension (which thus coincides with the map of Proposition~\ref{unico morfismo}).
\end{thm}
\begin{proof}
Observe that $\mathcal E_1$ is a singleton. By Lemma~\ref{eccesso}(\ref{eccesso3}), the monoid $\mathcal E$ satisfies all of the hypotheses of Proposition~\ref{universalita}. Hence we get the assertion.
\end{proof}

The following corollary follows directly from Theorem~\ref{teorema}.
\begin{cor}
The sets $\mathcal E_n$ of excessive multisegments of length $n$ are counted by Motzkin numbers.
\end{cor}

\section{Application to geometry of quiver representations}
\label{sez6}

In this section, we give a geometric application of the results obtained in the previous sections. We refer to Section~\ref{sez1} for definitions and notations regarding the geometric setting. Combining the main result of Fang and Reineke \cite{FR} with the tools developed in the previous sections,
 we prove that the supports  for the universal linear degenerations $\pi_n : X_n \rightarrow \mathcal{U}_{n}$ correspond to the excessive multisegments of length $n$. Hence, we provide an explicit characterization of the supports in terms of combinatorial/homological properties of the multisegments. As another application, we give a procedure to invert the parametrization of Fang and Reineke.

We first recall the result of Fang and Reineke \cite{FR}.
Given a multisegment $M \in \mathcal R_n$, we let $r(M)_{i,j}$, for $i,j \leq n$, be the number of segments in $M$ containing both $i$ and $j$. The multisegment $M$ is uniquely determined by the matrix $r(M)$ of the $r(M)_{i,j}$, which is called the {\em rank-tuple of $M$}.  
From \cite{FR}, one may give the following definition.
\begin{defn}
\label{FRmax}
We let  $FR: \mathcal{M}  ot \rightarrow \mathcal{M} $ be the map sending  $\gamma\in \mathcal{M}  ot_n$ to the unique multisegment  $FR(\gamma)\in \mathcal{M} _n$ whose rank-tuple satisfies
$$
r(FR(\gamma))_{i, j} = n+1 - \max_{i\leq k\leq l \leq m \leq j}\{\gamma(l)+\gamma(l-1)-\gamma(m)-\gamma(k-1)\},
$$
for all $i,j$ with $i\leq j$.
\end{defn}

The following is the main result of \cite{FR}.
\begin{thm}[Fang-Reineke]
\label{FR}
The supports are exactly the orbits corresponding to $Im (FR)$.
\end{thm}

Now we establish the explicit characterization of supports.
\begin{thm}
The map $FR$ is the unique graded monoid isoomorphism commuting with suspension.
Consequently,  $$Im (FR) = \mathcal{E} $$ and supports correspond to excessive multisegments.
\end{thm}

\begin{proof}
For readability,  let $r(\gamma)_{i,j}$ be the shorthand of $r(FR(\gamma))_{i,j}$, for $\gamma \in \mathcal M ot$.

By Definitions~\ref{3.1} and \ref{FR}, one can check that, if $M \in \mathcal{M} _p$, $N \in \mathcal{M} _q$, $\gamma \in \mathcal{M} ot_p$ and  $\eta\in \mathcal{M} ot_q $, then
$$
r(M\ast N)_{i,j} = \left\{
\begin{array}{ll}
r(M)_{i,j}+q  & \text{if   $1\leq i \leq j \leq p$};\\
r(N)_{i-p,j-p}+p  & \text{if   $p < i \leq j$};\\
\min\{r(M)_{i,p}+q, r(N)_{1,j-p}+p \}  & \text{if   $ i \leq p < j $};
\end{array}
\right. 
$$
and 
$$
r(\gamma\ast\eta)_{i,j} = \left\{
\begin{array}{ll}
r(\gamma)_{i,j}+q  & \text{if   $1\leq i \leq j \leq p$};\\
r(\eta)_{i-p,j-p}+p  & \text{if   $p < i \leq j$};\\
\min\{r(\gamma)_{i,p}+q, r(\eta)_{1,j-p}+p \}  & \text{if   $ i \leq p < j $};
\end{array}
\right. 
$$
for all $i,j$ with $i\leq j$.

By Definitions~\ref{sospensione} and \ref{FR}, one can check that, if $M \in \mathcal{M} $ and  $\gamma\in \mathcal{M} ot$, then
$$
r(S(M))_{i,j} = \left\{
\begin{array}{ll}
r(M)_{i-1,j-1}+2  & \text{if   $1 < i \leq j < p+2$};\\
r(M)_{1,j-1}+1  & \text{if   $1 = i \leq j < p+2$};\\
r(M)_{i-1,p}+1  & \text{if   $1 < i \leq j = p+2$};\\
r(M)_{1,p}  & \text{if   $1 = i, j = p+2$};\\
\end{array}
\right. 
$$
and
$$
r(S(\gamma))_{i,j} = \left\{
\begin{array}{ll}
r(\gamma)_{i-1,j-1}+2  & \text{if   $1 < i \leq j < p+2$};\\
r(\gamma)_{1,j-1}+1  & \text{if   $1 = i \leq j < p+2$};\\
r(\gamma)_{i-1,p}+1  & \text{if   $1 < i \leq j = p+2$};\\
r(\gamma)_{1,p}  & \text{if   $1 = i, j = p+2$};\\
\end{array}
\right. 
$$
for all $i,j$ with $i\leq j$.

Hence the first statement follows.

By Theorem~\ref{teorema} and the first statement, the map $FR$ is an isomorphism from $\mathcal{M}ot$ to $\mathcal E$. Now, the last statement follows by Theorem~\ref{FR}.
\end{proof}

As a final application, we consider the inverse $FR^{-1} : \mathcal E \mapsto \mathcal {M}ot$ of the Fang-Reineke map on its image. In this order of ideas,  note that, since $FR^{-1}(E_1\ast E_2)=FR^{-1}(E_1) \ast FR^{-1}( E_2)$ and  $FR^{-1}(S(E))= S(FR^{-1}(E))$, we have a recursive  procedure to compute $FR^{-1}$. Theorem~\ref{iniettivo} and Lemma~\ref{eccesso}  make this procedure very explicit. 

Given $E \in \mathcal E_n$, find its special full columns and write $E$ as the concatenation of primitive elements (which are restrictions by Remark~\ref{vice}). Each of these primitive elements is either the unique multisegment of length 1, or a suspension (of its double restriction, again by Remark~\ref{vice}). Then iterate these steps. 
As initial data, we have that the Motzkin paths corresponding to the empty multisegment (i.e., the identity element $id$ of the monoid $\mathcal E$) and to the unique multisegment of $\mathcal E_1$ (i.e., $\vcenter{\xymatrix@R=0pt{
*{\circ}\\
*{\circ}
}}$) are the unique Motzkin path of length 0 and the unique Motzkin path of length 1,  respectively.

\begin{ex} Let $M$ be the following multisegment in $\mathcal E_9$:
$$
M=\vcenter
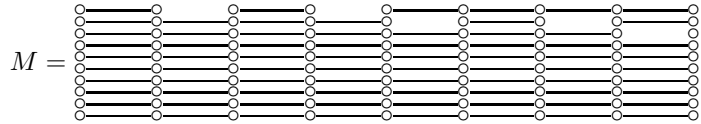
{\xymatrix@R=0pt{
*{\circ}\ar@{-}[r]&*{\circ}&*{\circ}\ar@{-}[r]&*{\circ}&*{\circ}\ar@{-}[r]&*{\circ}\ar@{-}[r]&*{\circ}\ar@{-}[r]&*{\circ}\ar@{-}[r]&*{\circ}\\
*{\circ}\ar@{-}[r]&*{\circ}\ar@{-}[r]&*{\circ}\ar@{-}[r]&*{\circ}\ar@{-}[r]&*{\circ}&*{\circ}\ar@{-}[r]&*{\circ}&*{\circ}\ar@{-}[r]&*{\circ}\\
*{\circ}\ar@{-}[r]&*{\circ}\ar@{-}[r]&*{\circ}\ar@{-}[r]&*{\circ}\ar@{-}[r]&*{\circ}\ar@{-}[r]&*{\circ}\ar@{-}[r]&*{\circ}\ar@{-}[r]&*{\circ}&*{\circ}\\
*{\circ}\ar@{-}[r]&*{\circ}\ar@{-}[r]&*{\circ}\ar@{-}[r]&*{\circ}\ar@{-}[r]&*{\circ}\ar@{-}[r]&*{\circ}\ar@{-}[r]&*{\circ}\ar@{-}[r]&*{\circ}\ar@{-}[r]&*{\circ}\\
*{\circ}\ar@{-}[r]&*{\circ}\ar@{-}[r]&*{\circ}\ar@{-}[r]&*{\circ}\ar@{-}[r]&*{\circ}\ar@{-}[r]&*{\circ}\ar@{-}[r]&*{\circ}\ar@{-}[r]&*{\circ}\ar@{-}[r]&*{\circ}\\
*{\circ}\ar@{-}[r]&*{\circ}\ar@{-}[r]&*{\circ}\ar@{-}[r]&*{\circ}\ar@{-}[r]&*{\circ}\ar@{-}[r]&*{\circ}\ar@{-}[r]&*{\circ}\ar@{-}[r]&*{\circ}\ar@{-}[r]&*{\circ}\\
*{\circ}\ar@{-}[r]&*{\circ}\ar@{-}[r]&*{\circ}\ar@{-}[r]&*{\circ}\ar@{-}[r]&*{\circ}\ar@{-}[r]&*{\circ}\ar@{-}[r]&*{\circ}\ar@{-}[r]&*{\circ}\ar@{-}[r]&*{\circ}\\
*{\circ}\ar@{-}[r]&*{\circ}\ar@{-}[r]&*{\circ}\ar@{-}[r]&*{\circ}\ar@{-}[r]&*{\circ}\ar@{-}[r]&*{\circ}\ar@{-}[r]&*{\circ}\ar@{-}[r]&*{\circ}\ar@{-}[r]&*{\circ}\\
*{\circ}\ar@{-}[r]&*{\circ}\ar@{-}[r]&*{\circ}\ar@{-}[r]&*{\circ}\ar@{-}[r]&*{\circ}\ar@{-}[r]&*{\circ}\ar@{-}[r]&*{\circ}\ar@{-}[r]&*{\circ}\ar@{-}[r]&*{\circ}\\
*{\circ}\ar@{-}[r]&*{\circ}\ar@{-}[r]&*{\circ}\ar@{-}[r]&*{\circ}\ar@{-}[r]&*{\circ}\ar@{-}[r]&*{\circ}\ar@{-}[r]&*{\circ}\ar@{-}[r]&*{\circ}\ar@{-}[r]&*{\circ}\\
}}
$$
The segment $M$ has two special full columns, the first and the third (observe that the sixth column is full but not special full). Let us first consider, for instance, the third column. We have
$$
M= \; _{(3)}M \ast M_{(6)}=\vcenter{\xymatrix@R=0pt{
*{\circ}\ar@{-}[r]&*{\circ}&*{\circ}\\
*{\circ}\ar@{-}[r]&*{\circ}\ar@{-}[r]&*{\circ}\\
*{\circ}\ar@{-}[r]&*{\circ}\ar@{-}[r]&*{\circ}\\
*{\circ}\ar@{-}[r]&*{\circ}\ar@{-}[r]&*{\circ}\\
}}
\ast
\vcenter{\xymatrix@R=0pt{
*{\circ}&*{\circ}\ar@{-}[r]&*{\circ}\ar@{-}[r]&*{\circ}\ar@{-}[r]&*{\circ}\ar@{-}[r]&*{\circ}\\
*{\circ}\ar@{-}[r]&*{\circ}&*{\circ}\ar@{-}[r]&*{\circ}&*{\circ}\ar@{-}[r]&*{\circ}\\
*{\circ}\ar@{-}[r]&*{\circ}\ar@{-}[r]&*{\circ}\ar@{-}[r]&*{\circ}\ar@{-}[r]&*{\circ}&*{\circ}\\
*{\circ}\ar@{-}[r]&*{\circ}\ar@{-}[r]&*{\circ}\ar@{-}[r]&*{\circ}\ar@{-}[r]&*{\circ}\ar@{-}[r]&*{\circ}\\
*{\circ}\ar@{-}[r]&*{\circ}\ar@{-}[r]&*{\circ}\ar@{-}[r]&*{\circ}\ar@{-}[r]&*{\circ}\ar@{-}[r]&*{\circ}\\
*{\circ}\ar@{-}[r]&*{\circ}\ar@{-}[r]&*{\circ}\ar@{-}[r]&*{\circ}\ar@{-}[r]&*{\circ}\ar@{-}[r]&*{\circ}\\
*{\circ}\ar@{-}[r]&*{\circ}\ar@{-}[r]&*{\circ}\ar@{-}[r]&*{\circ}\ar@{-}[r]&*{\circ}\ar@{-}[r]&*{\circ}\\
}}
$$
Considering the special full column of the first factor above, we write $M$ as the concatenation of three segments with no special full columns (hence primitive):
$$
M= \vcenter{\xymatrix@R=0pt{
*{\circ}\\
*{\circ}
}}
\ast\vcenter{\xymatrix@R=0pt{
*{\circ}&*{\circ}\\
*{\circ}\ar@{-}[r]&*{\circ}\\
*{\circ}\ar@{-}[r]&*{\circ}\\
}}
\ast
\vcenter{\xymatrix@R=0pt{
*{\circ}&*{\circ}\ar@{-}[r]&*{\circ}\ar@{-}[r]&*{\circ}\ar@{-}[r]&*{\circ}\ar@{-}[r]&*{\circ}\\
*{\circ}\ar@{-}[r]&*{\circ}&*{\circ}\ar@{-}[r]&*{\circ}&*{\circ}\ar@{-}[r]&*{\circ}\\
*{\circ}\ar@{-}[r]&*{\circ}\ar@{-}[r]&*{\circ}\ar@{-}[r]&*{\circ}\ar@{-}[r]&*{\circ}&*{\circ}\\
*{\circ}\ar@{-}[r]&*{\circ}\ar@{-}[r]&*{\circ}\ar@{-}[r]&*{\circ}\ar@{-}[r]&*{\circ}\ar@{-}[r]&*{\circ}\\
*{\circ}\ar@{-}[r]&*{\circ}\ar@{-}[r]&*{\circ}\ar@{-}[r]&*{\circ}\ar@{-}[r]&*{\circ}\ar@{-}[r]&*{\circ}\\
*{\circ}\ar@{-}[r]&*{\circ}\ar@{-}[r]&*{\circ}\ar@{-}[r]&*{\circ}\ar@{-}[r]&*{\circ}\ar@{-}[r]&*{\circ}\\
*{\circ}\ar@{-}[r]&*{\circ}\ar@{-}[r]&*{\circ}\ar@{-}[r]&*{\circ}\ar@{-}[r]&*{\circ}\ar@{-}[r]&*{\circ}\\
}}
$$
The first factor is the unique multisegment in $\mathcal E_1$. The second is the suspension of the empty multisegment $id$. The third is the suspension of the multisegment
$$
N= \vcenter{\xymatrix@R=0pt{
*{\circ}\ar@{-}[r]&*{\circ}\ar@{-}[r]&*{\circ}\ar@{-}[r]&*{\circ}\\
*{\circ}&*{\circ}\ar@{-}[r]&*{\circ}&*{\circ}\\
*{\circ}\ar@{-}[r]&*{\circ}\ar@{-}[r]&*{\circ}\ar@{-}[r]&*{\circ}\\
*{\circ}\ar@{-}[r]&*{\circ}\ar@{-}[r]&*{\circ}\ar@{-}[r]&*{\circ}\\
*{\circ}\ar@{-}[r]&*{\circ}\ar@{-}[r]&*{\circ}\ar@{-}[r]&*{\circ}\\
}}
$$
The multisegment $N$ has a special full column and satisfies
$$
N=\vcenter{\xymatrix@R=0pt{
*{\circ}&*{\circ}\\
*{\circ}\ar@{-}[r]&*{\circ}\\
*{\circ}\ar@{-}[r]&*{\circ}\\
}}
\ast
\vcenter{\xymatrix@R=0pt{
*{\circ}&*{\circ}\\
*{\circ}\ar@{-}[r]&*{\circ}\\
*{\circ}\ar@{-}[r]&*{\circ}\\
}}
$$
As already noted, both factors of $N$ are the suspension of the empty segment $id$. 
Hence
$$
M= \vcenter{\xymatrix@R=0pt{
*{\circ}\\
*{\circ}
}}
\ast
S( id )
\ast
S\big( S(id) \ast S(id) \big)
$$

and 
$$
FR^{-1}(M)= 
\vcenter{\xymatrix@C=2pt@R=2pt{
&&&&&&&&&&&\\
&{\circ}&{\circ}&{\circ}&{\circ}&{\circ}&*-{\bullet}&{\circ}&*-{\bullet}&{\circ}&{\circ}\\
&{\circ}&{\circ}&*-{\bullet}&{\circ}&*-{\bullet}&{\circ}&*-{\bullet}&{\circ}&*-{\bullet}&{\circ}\\
\ar@{-}[rrrrrrrrrr] 
\ar@{-}'[rr]'[rrru]'[rrrr]'[rrrrru]'[rrrrrruu]'[rrrrrrru]'[rrrrrrrruu]'[rrrrrrrrru]'[rrrrrrrrrr]&
*-{\bullet}&*-{\bullet}&{\circ}&*-{\bullet}&{\circ}&{\circ}&{\circ}&{\circ}&{\circ}&*-{\bullet}& &\\
&\ar@{-}[uuuu]&&&&&&&&&&
}}
$$
\end{ex}

\end{document}